\documentclass[12pt,a4paper]{amsart}
\usepackage{amsmath}
\usepackage{srcltx}
\usepackage{fullpage}

\theoremstyle{plain}
\newtheorem{thm}{Theorem}[section]
\newtheorem{lem}[thm]{Lemma}
\theoremstyle{definition}

\newtheorem{ex}[thm]{Example}

\numberwithin{equation}{section}

\begin{document}
\title[A nonlinear nonlocal impulsive system]{Nonnegative solutions for a system of impulsive BVPs with nonlinear nonlocal BCs}
\date{}


\author{Gennaro Infante}
\address{Gennaro Infante, Dipartimento di Matematica e Informatica, Universit\`{a} della
Calabria, 87036 Arcavacata di Rende, Cosenza, Italy}%
\email{gennaro.infante@unical.it}%

\author{Paolamaria Pietramala}%
\address{Paolamaria Pietramala, Dipartimento di Matematica e Informatica, Universit\`{a} della
Calabria, 87036 Arcavacata di Rende, Cosenza, Italy}%
\email{pietramala@unical.it}%

\begin{abstract}
We study the existence of nonnegative solutions for
a system of impulsive differential equations subject to nonlinear, nonlocal boundary conditions. The system presents a coupling in the differential equation and in the boundary conditions. 
 The main tool that we use is the theory of fixed point index for compact maps.
\end{abstract}

\maketitle

\section{Introduction}
The aim of this paper is to study the existence and multiplicity of positive solutions for a class of systems of ordinary impulsive differential equations subject to nonlinear, nonlocal boundary conditions (BCs). The system presents a coupling in the nonlinearities and in the BCs. Problems with a coupling in the BCs often occur in applications, see for example~\cite{Amann, Asif-khan-jmma, cui-sun, Goodrich1, Goodrich2, gipp-mmas, gifmpp-cnsns, Leung, meh-nic, sun, yjoa}. 
On the other hand, impulsive problems have been studied not only because of a theoretical interest, but also 
because they model several phenomena in engineering, physics and life sciences. 
For example, Nieto and co-authors \cite{nieto1, nieto2} contributed to the field of population dynamics.
An introduction to the theory of impulsive differential equations and its applications can be found in the books
\cite{bainov,be-he-nt-book,lak-bai-sim,sam-per}.

Systems of second order impulsive boundary value problems (BVPs) have been studied in \cite{lele, lihuwu, raba, suchnio}. Here we consider the (fairly general) system of second order differential equations of the form
\begin{equation}\label{ODE}
\begin{array}{c}
u''(t) + g_1(t)f_1(t, u(t),v(t)) = 0,\ t \in (0,1),\ t\neq \tau_1,\\
v''(t) + g_2(t)f_2(t, u(t),v(t)) = 0,\ t \in (0,1),\ t\neq \tau_2,
\end{array}
\end{equation}
with impulsive terms of the type
\begin{equation}\label{imp}
\begin{array}{c}
\Delta u|_{t=\tau_1}=I_1(u(\tau_1)), \;\,\Delta u'|_{t=\tau_1}=N_1(u(\tau_1)), {\tau_1}\in (0,1),\\
\Delta v|_{t=\tau_2}=I_2(v(\tau_2)), \;\,\Delta v'|_{t=\tau_2}=N_2(v(\tau_2)), {\tau_2}\in (0,1),
\end{array}
\end{equation}
and nonlocal nonlinear BCs of `Sturm-Liouville' kind
\begin{equation}\label{SL}
\begin{array}{c}
a_{11}u(0)-b_{11}u'(0)=H_{1}(\alpha_{1}[u]), \;\;a_{12}u(1)+b_{12}u'(1)=L_{1}(\beta_{1}[v]),\\
a_{21}v(0)-b_{21}v'(0)=H_{2}(\alpha_{2}[v]), \;\;a_{22}v(1)+b_{22}v'(1)=L_{2}(\beta_{2}[u]),
\end{array}
\end{equation}
where
for $i=1,2$, $a_{i1},b_{i1},a_{i2},b_{i2}\in [0,\infty)$, $a_{i1}+b_{i1}\neq 0, a_{i2}+b_{i2}\neq 0$ and $\lambda=0$ is not an eigenvalue of the problem
$$
w''(t)=0, \ a_{i1}w(0)-b_{i1}w'(0)=0,
\;a_{i2}w(1)+b_{i2}w'(1)=0.
$$
Here $\Delta w|_{t=\tau}$ denotes the ``jump'' of the function $w$ in $t=\tau$, that is
$$
\Delta w|_{t=\tau}=w(\tau^+)-w(\tau^-),
$$
where $w(\tau^-)$ and $w(\tau^+)$ are the left and right limits of $w$ in $t=\tau$  and  $\alpha_{i}[\cdot]$, $\beta_{i}[\cdot]$ are bounded linear functionals 
given by positive Riemann-Stieltjes integrals, namely
\begin{equation*}
\alpha _{i}[w]=\int_{0}^{1}w(s)\,dA_{i}(s),\,\,\,\,\beta _{i}[w]=\int_{0}^{1}w(s)\,dB_{i}(s).
\end{equation*}%
This type of formulation includes, as special cases, multi-point or integral conditions, namely
\begin{equation*}
\alpha_{i}[w]=\sum_{j=1}^{m} \alpha_{ij}w(\eta_{ij})\ \text{and}\
\alpha_{i}[w]=\int_{0}^{1} {\alpha}_{i}(s)w(s)\,ds,
\end{equation*}
studied for example~\cite{dovo, hen-luca, gippmt,  Jank1,Jank2,kttmna,  rma, sotiris,sapa1, sapa2,  jw-jlms, jw-gi-jlmsII}. In the case of impulsive equations, nonlocal BCs have been studied by many authors, see for example \cite{be-ber-he1, begagont, ongipp, fe-du-ge, feng-xie, John1, Jank1,Jank2,liu-ge,xia-reg} and
references therein.
The functions  $H_{i}$, $L_{i}$ are continuous functions;  for earlier contributions on problems with nonlinear BCs we refer the reader to~\cite{ Cabada1, acfm-nonlin,  Goodrich1, Goodrich2, gi-caa, gipp-cant, lmpp, paola} and references therein.

Our idea is to start from  the results of \cite{gipp-mmas, gifmpp-cnsns}, valid for non-impulsive systems,  and to
rewrite the system \eqref{ODE}-\eqref{SL} as a system of perturbed Hammerstein integral equations, namely
\begin{gather*}
\begin{aligned}
u(t)=\gamma_{1}(t)H_{1}(\alpha_{1} [u])+\delta_1(t)L_{1}(\beta_{1}
[v])+\int_{0}^{1}k_1(t,s)g_1(s)f_1(s,u(s),v(s))\,ds+ G_1(u)(t), \\
v(t)= \gamma_{2}(t)H_{2}(\alpha_{2}
[v])+\delta_2(t)L_{2}(\beta_{2} [u])+\int_{0}^{1}k_2(t,s)g_2(s)f_2(s,u(s),v(s))\,ds+ G_2(v)(t),%
\end{aligned}
\end{gather*}
where the functions $\gamma_i, \delta_i$ are the unique solutions of
\begin{align*}
&{\gamma_i}''(t)=0, \;\;\; a_1{\gamma_i}(0)-b_1{\gamma_i}'(0)=1,
\;a_2{\gamma_i}(1)+b_2{\gamma_i}'(1)=0,\\
&{\delta_i}''(t)=0, \;\;\;\; a_1{\delta_i}(0)-b_1{\delta_i}'(0)=0,
\;\, a_2{\delta_i}(1)+b_2{\delta_i}'(1)=1,
\end{align*}
and the functions $G_i$, that are construct in natural manner, take care of the impulses. 

Systems of perturbed Hammerstein integral equations were studied in \cite{df-gi-do, Goodrich1, Goodrich2,gipp-ns, gipp-nonlin, gifmpp-cnsns, kang-wei, ya1}. Our existence theory for multiple positive solutions of the perturbed Hammerstein integral
equations  covers the system \eqref{ODE}-\eqref{SL} as a special case and we show in an example that all the constants that occur in our theory can be computed.  Here we focus on \emph{positive} measures, because we want our functionals to preserve some inequalities. Our methodology involves the construction of \emph{new} Stieltjes measures that take into account  the boundary conditions and the impulsive effect.

We make use of the classical fixed point index theory (see for example \cite{amann, guolak}) and also benefit of ideas from the papers
\cite{gi-caa, gipp-nonlin, gipp-mmas, gijwems,  gifmpp-cnsns, gippmz, jwgi-lms}.

\section{The System of Integral Equations}
We begin with the assumptions on the terms that occur in the system of perturbed Hammerstein integral equations
\begin{gather}
\begin{aligned}\label{syst}
u(t)=\gamma_{1}(t)H_{1}(\alpha_{1} [u])+\delta_1(t)L_{1}(\beta_{1}
[v])+G_1(u)(t)+F_1(u,v)(t), \\
v(t)= \gamma_{2}(t)H_{2}(\alpha_{2}
[v])+\delta_2(t)L_{2}(\beta_{2} [u])+G_2(v)(t)+F_2(u,v)(t),%
\end{aligned}
\end{gather}
where
\begin{equation}
F_i(u,v)(t):=\int_{0}^{1}k_i(t,s)g_i(s)f_i(s,u(s),v(s))\,ds.
\end{equation}
The functions $G_i$ are given, as in \cite{gippmz},  by
\begin{equation}
G_i(w)(t):=\gamma_{i}(t)\chi_{(\tau_i,1]}(d_{i1}I_i+e_{i1}N_i)(w(\tau_i))+\delta_i(t)\chi_{[0,\tau_i]}(d_{i2}I_i+e_{i2}N_i)(w(\tau_i)),
\end{equation}
 with  coefficients
 $$
  d_{i1}=\frac{\delta_i'(\tau_i)}{W_i(\tau_i)},\ e_{i1}=\frac{-\delta_i(\tau_i)}{W_i(\tau_i)},\ d_{i2}=\frac{\gamma_i'(\tau_i)}{W_i(\tau_i)}
  \text{ and  }e_{i2}=\frac{-\gamma_i(\tau_i)}{W_i(\tau_i)},
 $$
 where $W_i$ is the Wronskian, $W_i(t)=\gamma_i(t)\delta_i'(t)-\delta_i(t)\gamma_i'(t)$.
 
We assume that for every $i=1,2$,
\begin{itemize}
\item  $f_i: [0,1]\times [0,\infty)\times [0,\infty) \to
[0,\infty)$ satisfies Carath\'{e}odory conditions, that is, $f_i(\cdot,u,v)$
is measurable for each fixed $(u,v)$ and $f_i(t,\cdot,\cdot)$ is continuous
for almost every (a.e.) $t\in [0,1]$, and for each $r>0$ there exists $\phi_{i,r} \in
L^{\infty}[0,1]$ such that{}
\begin{equation*}
f_i(t,u,v)\le \phi_{i,r}(t) \;\text{ for } \; u,v\in [0,r]\;\text{ and
a.\,e.} \; t\in [0,1].
\end{equation*}%
{}
\item $k_i:[0,1]\times [0,1]\to [0,\infty)$ is
measurable, and for every $\tau\in [0,1]$ we have
\begin{equation*}
\lim_{t \to \tau} |k_i(t,s)-k_i(\tau,s)|=0 \;\text{ for a.\,e.}\, s \in [0,1].
\end{equation*}%
{}
\item  there exist a subinterval $[a_i,b_i] \subseteq
(\tau_i,1] $, a function $\Phi_i \in L^{\infty}[0,1]$, and a constant $c_{\Phi_i} \in
(0,1]$, such that
\begin{align*}
k_i(t,s)\leq \Phi_i(s) \text{ for } &t \in [0,1] \text{ and a.\,e.}\, s\in [0,1], \\
k_i(t,s) \geq c_{\Phi_i}\Phi_i(s) \text{ for } &t\in [a_i,b_i] \text{ and a.\,e.} \, s \in [0,1].
\end{align*}%
{}
\item $g_i\,\Phi_i \in L^1[0,1]$, $g_i \geq 0$ a.e., and $%
\int_{a_i}^{b_i} \Phi_i(s)g_i(s)\,ds >0$.
{}
\item $\alpha _{i}[\cdot ]$ and $\beta _{i}[\cdot ]$ are linear functionals
given by
\begin{equation*}
\alpha _{i}[w]=\int_{0}^{1}w(s)\,dA_{i}(s),\,\,\,\,\beta _{i}[w]=\int_{0}^{1}w(s)\,dB_{i}(s),
\end{equation*}%
involving Riemann-Stieltjes integrals; $A_{i}$ and $B_{i}$ are of bounded variation and continuous in $\tau_i$ and $%
dA_{i},dB_{i}$ are \emph{positive} measure.
\item  $H_{i},L_{i}: [0,\infty)\to [0,\infty)$ are continuous functions such that
there exist
$h_{i1}, h_{i2}, l_{i2} \in [0,\infty)$, 
with
\begin{equation*}
 h_{i1}w \leq H_{i}(w)\leq  h_{i2}w,\,\,\,\,  L_{i}(w)\leq  l_{i2}w,
\end{equation*}
for every $w\geq 0$.
\item $\gamma_{i},\delta_i \in C[0,1], \;\gamma_{i},\delta_i \geq 0$,
  and there exist $c_{\gamma_i} ,c_{\delta_i}\in(0,1]$ such that
\begin{equation*}
\gamma _{i}(t)\geq c_{\gamma_i}\| \gamma _{i}\| _{\infty },\,\,\delta _{i}(t)\geq c_{\delta_i}\| \delta _{i}\| _{\infty }\;\text{for every%
}\;t\in [ a_{i},b_{i}],
\end{equation*}%
where $\| w\| _{\infty }:=\sup \{|w(t)|,\;t\;\in [ 0,1]\}$.
\item $I_i, N_i:[0,\infty)\rightarrow \mathbb{R}$  are continuous functions and there exist $p_{i11},p_{i12},q_{i11}>0$ and $p_{i22}\ge 0$ such that for $w\in [0,\infty)$
 $$
 p_{i11} w\leq (d_{i1}I_i+e_{i1}N_i)(w) \leq p_{i12} w,
 $$
 and
 $$
 0\leq (d_{i2}I_i+e_{i2}N_i)(w) \leq p_{i22} w.
$$
\end{itemize}

We consider the Banach space
 \begin{align*}
 PC_{\tau}[0,1]:=\{w:&[0,1]\rightarrow\mathbb{R}, \; w \;\text{is continuous in}\; t\in [0,1]\backslash\{\tau\},\\
 &\text{there exist}\; w(\tau^-)=w(\tau)\;\text{and}\;|w(\tau^+)|<\infty \},
 \end{align*}
 endowed with the supremum norm $\|\cdot\| _{\infty }$.

We work in the space $PC_{\tau_1}[0,1]\times PC_{\tau_2}[0,1]$ endowed with the norm
\begin{equation*}
\| (u,v)\| :=\max \{\| u\| _{\infty },\| v\| _{\infty }\}.
\end{equation*}%
Let
\begin{equation*}
\tilde{K_{i}}:=\{w\in PC_{\tau_i}[0,1]:w(t)\geq 0\ \text{for}\ t\in [ 0,1]\,\,%
\text{and}\,\,\min_{t\in [ a_{i},b_{i}]}w(t)\geq c_{i}\|
w\| _{\infty }\},
\end{equation*}%
where 
$$
c_i=\min\Bigl\{c_{\Phi_i},c_{\gamma_i},c_{\delta_i},\dfrac{c_{\gamma_i}\|\gamma_i\|_{\infty}p_{i11}}{\max\{\|\gamma_i\|p_{i12}, \|\delta_i\|p_{i22}\}}\Bigr\}
$$ and consider the cone $K$
in $PC_{\tau_1}[0,1]\times PC_{\tau_2}[0,1]$ defined by
\begin{equation*}
\begin{array}{c}
K:=\{(u,v)\in \tilde{K_{1}}\times \tilde{K_{2}}\}.%
\end{array}%
\end{equation*}
For a \emph{positive} solution of the system \eqref{syst} we mean a solution
$(u,v)\in K$ of \eqref{syst} such that $\|(u,v)\|>0$.

We now  show that the integral operator
\begin{gather}
\begin{aligned}    \label{opT}
T(u,v)(t):=& 
\left(
\begin{array}{c}
\gamma_{1}(t)H_{1}(\alpha_{1} [u])+\delta_1(t)L_{1}(\beta_{1}
[v])+G_1(u)(t)+F_1(u,v)(t) \\
\gamma_{2}(t)H_{2}(\alpha_{2}
[v])+\delta_2(t)L_{2}(\beta_{2} [u])+G_2(v)(t)+F_2(u,v)(t)%
\end{array}
\right)
\\ &:=
\left(
\begin{array}{c}
T_1(u,v)(t) \\
T_2(u,v)(t)%
\end{array}
\right) ,
\end{aligned}
\end{gather}
leaves the cone $K$ invariant and is compact.
In order to do this, we use  the following compactness criterion, which can be found, for example, in
\cite{lak-bai-sim} and is an extension of the classical Ascoli-Arzel\`{a} Theorem.

\begin{lem} \label{dan}
A set $S\subseteq PC_\tau[0,1]$ is relatively compact in $PC_\tau[0,1]$ if
and only if $S$ is bounded and quasi-equicontinuous (i.e. $\forall u\in S$ and $\forall \varepsilon>0$,  $\exists \beta>0$ such that $t_1,t_2 \in [0,\tau]$ (or $t_1,t_2 \in (\tau,1]$) and $|t_1-t_2|<\beta$
implies $|u(t_1)-u(t_2)|<\varepsilon$).
\end{lem}

\begin{lem}
The operator \eqref{opT} maps $K$ into $K$ and is compact.
\end{lem}

\begin{proof}
Take $(u,v)\in K$ such that $\| (u,v)\| \leq r$. Then  we have, for $%
t\in [ 0,1]$,
\begin{equation*}
\Lambda_{1}(u,v)(t):=\gamma_{1}(t)H_{1}(\alpha_{1} [u])+\delta_1(t)L_{1}(\beta_{1}[v])+\int_{0}^{1}k_{1}(t,s)g_{1}(s)f_{1}(s,u(s),v(s))\,ds
\end{equation*}%
and therefore
\begin{equation*}
\| \Lambda_{1}(u,v)\|_{\infty } \leq \| \gamma _{1}\| _{\infty }H_{1}(\alpha_{1} [u])+\| \delta _{1}\| _{\infty }L_{1}(\beta_{1}
[v])+\int_{0}^{1}\Phi_{1}(s)g_{1}(s)f_{1}(s,u(s),v(s))\,ds.
\end{equation*}%
We obtain,  as in Lemma 1 of~\cite{gifmpp-cnsns}, 
\begin{align*}
\min_{t\in [ a_{1},b_{1}]}\Lambda_{1}(u,v)(t) \geq &c_{\gamma_1}\| \gamma _{1}\| _{\infty }H_{1}(\alpha_{1} [u])+c_{\delta_1}\| \delta _{1}\| _{\infty }L_{1}(\beta_{1}[v])\\
&+c_{\Phi_1}\int_{0}^{1}\Phi _{1}(s)g_{1}(s)f_{1}(s,u(s),v(s))\,ds 
\geq \min\{c_{\Phi_i},c_{\gamma_i},c_{\delta_i}\}\| \Lambda_{1}(u,v)\|_{\infty } .
\end{align*}%

On the other hand, for $t\in [0,\tau_1]$ we have
$$
G_1(u)(t)\le \|\delta_1\|_{\infty } p_{122}u(\tau_1)
$$
and for $t\in (\tau_1,1]$
$$
G_1(u)(t)\le \|\gamma_1\|_{\infty } p_{112}u(\tau_1).
$$
Therefore for $t\in[0,1]$ we obtain
\begin{equation*}
  G_1(u)(t)\le u(\tau_1)\max\{\|\gamma_1\|_{\infty }p_{112}, \|\delta_1\|_{\infty }p_{122}\}
\end{equation*}
and thus
\begin{equation*}
\|G_1(u)\|\le u(\tau_1)\max\{\|\gamma_1\|_{\infty }p_{112}, \|\delta_1\|_{\infty }p_{122}\}.
\end{equation*}
For $t\in [a_1,b_1]$, we get
$$
\aligned
G_1(u)(t)&= \gamma_1(t)(d_{11}I_1+e_{11}N_1)(u(\tau_1))\\
                      &\geq\dfrac{c_{\gamma_1}\|\gamma_1\|_{\infty}p_{111}}{\max\{\|\gamma_1\|_{\infty }p_{112}, \|\delta_1\|_{\infty }p_{122}\}}u(\tau_1)\max\{\|\gamma_1\|_{\infty }p_{112}, \|\delta_1\|_{\infty }p_{122}\}.
\endaligned
$$
Thus we obtain
$$
\min_{t\in [a_1,b_1]}T_1(u,v)(t) \ge c_1\|T_1(u,v)\|_{\infty }.
$$
Moreover, we have $T_1(u,v)(t)\geq 0$. Hence we have $T_{1}(u,v)\in \tilde{K_{1}}$. In a similar manner we proceed for $ T_{2}(u,v)$.

Furthermore, the map $T$ is compact since the components $T_{i}$ are sum of  compact maps: the compactness of $F_{i}$ is well-known; the compactness of the term $G_i$ follows, in a similar way as in ~\cite{gippmz}, from Lemma~\ref{dan}; since $\gamma
_{i},\delta _{i},H _{i},L _{i}$ are continuous, the remaining terms map bounded sets into bounded subsets of a finite dimensional space.
\end{proof}

\section{Fixed point index calculations}

\subsection{Preliminaries and notations}
We  recall some basic facts regarding the classical fixed point index for compact maps, see for example  \cite{amann, guolak}. 

Let $K$ be a cone in a Banach space $X$. If $\Omega$ is a bounded open subset of $K$ (in the relative
topology) we denote by $\overline{\Omega}$ and $\partial \Omega$
the closure and the boundary relative to $K$. When $\Omega$ is an open
bounded subset of $X$ we write $\Omega_{K}=\Omega \cap K$, an open subset of $K$.
\begin{thm}
Let $K$ be a cone in a Banach space $X$ and let $\Omega$ be an open bounded set with $0\in \Omega_{K}$ and
$\overline{\Omega}_{K}\ne K$. Assume that $T:\overline{\Omega}_{K}\to K$ is a compact map such that $x\neq Tx$ for $x\in \partial \Omega_{K}$. Then
the fixed point index $i_{K}(T, \Omega_{K})$ has the following
properties.
\begin{itemize}
\item[(1)] If there
exists $e\in K\setminus \{0\}$ such that $x\neq Tx+\mu e$ for
all $x\in \partial \Omega_{K}$ and all $\mu
\geq 0$, then $i_{K}(T, \Omega_{K})=0$.
\item[(2)] If $Tx \neq \mu x$ for all $x\in
\partial \Omega_{K}$ and all $\mu \geq 1$, then $i_{K}(T, \Omega_{K})=1$.
\item[(3)] Let $\Omega^{1}$ be open in $X$ with
$\overline{\Omega_{K}^{1}}\subset \Omega_{K}$. If $i_{K}(T, \Omega_{K})=1$ and
$i_{K}(T, \Omega_{K}^{1})=0$, then $T$ has a fixed point in
$\Omega_{K}\setminus \overline{\Omega_{K}^{1}}$. The same result holds if
$i_{K}(T, \Omega_{K})=0$ and $i_{K}(T, \Omega_{K}^{1})=1$.
\end{itemize}
\end{thm}

For our index calculations, we use  the following (relative) open bounded sets in $K$:
\begin{equation*}
K_{\rho} = \{ (u,v) \in K : \|(u,v)\|< \rho \},
\end{equation*}
and
\begin{equation*}
V_\rho=\{(u,v) \in K: \min_{t\in [a_1,b_1]}u(t)<\rho\ \text{and}\ \min_{t\in
[a_2,b_2]}v(t)<\rho\}.
\end{equation*}
The set $V_\rho$ (in the context of systems) was introduced by the authors in~\cite{gipp-ns} and is equal to the set
called $\Omega^{\rho /c}$ in~\cite{df-gi-do}.
From now on we set 
$$c=\min\{{c_1},{c_2}\}.$$
We utilize  the following Lemma, the proof is similar to Lemma $5$ of \cite{df-gi-do} and is omitted.

\begin{lem}  
 The sets $K_{\rho}$ and $V_{\rho}$  have the following properties:
\begin{itemize}
\item[-] $K_{\rho}\subset V_{\rho}\subset K_{\rho/c}$.
\item[-] $(w_1,w_2) \in \partial V_\rho$ \; iff \; $(w_1,w_2)\in K$ and $\displaystyle
\min_{t\in [a_i,b_i]} w_i(t)= \rho$ for some $i\in \{1,2\}$ and $\displaystyle
\min_{t\in [a_i,b_i]}w_i(t)
\le \rho$ for each $i\in \{1,2\}$.
\item[-] If $(w_1,w_2) \in \partial V_\rho$, then for some $i\in\{1,2\}$ $\rho \le w_i(t) \le \rho/c$
 for each $t \in [a_i,b_i]$ and for each $i\in \{1,2\}$ we have $0 \leq w_i(t) \leq \rho/c$ for each $t\in [a_i,b_i]$ and $\|w_i\|_\infty \leq \rho/c$.
\end{itemize}
\end{lem}

We  introduce, in a similar way as in \cite{gi-pp-san}, the linear functionals
\begin{align*}
\tilde{\alpha}_i[w]:=&h_{i2}\alpha_i[w]+p_{i12} w(\tau_i):= \int_0^1 w(s)\,d \tilde{A}_i(s),\,\,i=1,2,\\
\bar{\alpha}_i[w]:=&h_{i1}\alpha_i[w]+p_{i11} w(\tau_i):= \int_0^1 w(s)\,d \bar{A}_i(s),\,\,\,i=1,2,\\
\end{align*}
and, for a measure $dC$, we use the notation 
$$
\mathcal{K}^i_{C}(s):=\int_0^1 k_i(t,s) \,dC(t).
$$
We assume from now on that
\begin{itemize}
\item $ \tilde{\alpha}_1[\gamma_1] <1,\; \text{and}\, \tilde{\alpha}_2[\gamma_2]<1.$
\end{itemize}

\subsection{Index on the set $K_{\rho}$}
We    prove a result concerning the fixed point index on the set $K_{\rho}$.

\begin{lem} 
Assume that
\begin{enumerate}
\item[$(\mathrm{I}_{\protect\rho }^{1})$]  there exists $\rho >0$
such that for every $i=1,2$
\begin{equation}
\Bigl(\frac{ \|\gamma _{i} \| _{\infty }\tilde{\alpha}_i[\delta _{i}]}{1-\tilde{\alpha}_i[\gamma _{i}]}+ \| \delta _{i} \| _{\infty
}\Bigr)(l_{i2}\beta_i[1]+p_{i22})+f_{i}^{0,\rho }\Bigl(\dfrac{1}{m_{i}}+\dfrac{ \| \gamma
_{i} \| _{\infty }}{1-\tilde{\alpha} _{i}[\gamma _{i}]}\int_{0}^{1}\mathcal{K}^i
_{\tilde{A_i}}(s)g_{i}(s)\,ds\Bigr)<1,  \label{eqmestt}
\end{equation}%
where
\begin{equation*}
f_{i}^{0,{\rho }}=\sup \Bigl\{\frac{f_{i}(t,u,v)}{\rho }:\;(t,u,v)\in
[ 0,1]\times [ 0,\rho ]\times [ 0,\rho ]\Bigr\}\ \text{and}%
\ \frac{1}{m_{i}}=\sup_{t\in [ 0,1]}\int_{0}^{1}k_{i}(t,s)g_{i}(s)\,ds.
\end{equation*}%
{}
\end{enumerate}

Then  $i_{K}(T,K_{\rho})=1$.
\end{lem}
\begin{proof}
 We show that $T(u,v) \neq \mu (u,v)$ for all $\mu \geq 1$ when $(u,v) \in \partial K_{\rho}$;
 this ensures, that the index is 1 on $K_{\rho}$. In fact, if this is not so, then there exist $(u,v)\in K$ with
$\|(u,v)\|=\rho$ and $\mu \geq 1$ such that $\mu (u,v)(t)=T(u,v)(t)$.
 Assume, without loss of
generality, that $ \| u \| _{\infty }=\rho $ and $ \| v \| _{\infty
}\leq \rho $. 
 We have for $t\in[0,1]$
\begin{align*}
 \mu u(t)=&\gamma_1(t)(H_1(\alpha_1[u])+\chi_{(\tau_1,1]}(d_{11}I_1+e_{11}N_1)(u(\tau_1)))\\
&+\delta_1(t)(L_1(\beta_1[v])+\chi_{[0,\tau_1]}(d_{12}I_1+e_{12}N_1)(u(\tau_1)))+ F_1(u,v)(t).
\end{align*}
 Since
$$
\tilde{\alpha}_1[u]\geq H_1(\alpha_1[u])+(d_{11}I_1+e_{11}N_1)(u(\tau_1)),
$$
we obtain 
\begin{equation*}
\mu u(t)\leq\gamma_1(t)\tilde{\alpha}_1[u]
+\delta_1(t)(l_{12}\beta_1[v]+(d_{12}I_1+e_{12}N_1)(u(\tau_1)))+ F_1(u,v)(t),
\end{equation*}
and moreover, since $v(t)\leq \rho$ and $u(t)\leq \rho$ for all $t\in [ 0,1]$, we obtain
\begin{align}
\mu u(t) \leq &\gamma_1(t)\tilde{\alpha}_1[u]
+\delta_1(t)(l_{12}\beta_1[\rho]+p_{122}u(\tau_1))+ F_1(u,v)(t)  \label{dis2} \\
\leq &\gamma_1(t)\tilde{\alpha}_1[u]
+\delta_1(t) \rho(l_{12}\beta_1[1]+p_{122})+ F_1(u,v)(t).  \notag
\end{align}
Applying $\tilde{\alpha}_1$ to both sides of \eqref{dis2} gives
\begin{equation*}
\mu \tilde{\alpha}_1[u] \leq 
\tilde{\alpha}_1[\gamma_1]\tilde{\alpha}_1[u]
+\tilde{\alpha}_1[\delta_1] \rho(l_{12}\beta_1[1]+p_{122})+\tilde{\alpha}_1[F_1(u,v)].  
\end{equation*}%
Thus we have
\begin{equation*}
(\mu -\tilde{\alpha}_1[\gamma_1])\tilde{\alpha}_1[u] \leq \tilde{\alpha}_1[\delta_1] \rho(l_{12}\beta_1[1]+p_{122})+\tilde{\alpha}_1[F_1(u,v)],
\end{equation*}%
that is
\begin{equation*}
\tilde{\alpha}_1[u]\leq \rho \frac{\tilde{\alpha}_1[\delta_1] (l_{12}\beta_1[1]+p_{122})}{
\mu -\tilde{\alpha}_1[\gamma_1]}+\frac{\tilde{\alpha}_1[F_1(u,v)]}{\mu -\tilde{\alpha}_1[\gamma_1]}.
\end{equation*}

Substituting into \eqref{dis2} gives
\begin{align*}
\mu u(t) \leq &\gamma _{1}(t)\Bigl(\rho \frac{\tilde{\alpha}_1[\delta_1] (l_{12}\beta_1[1]+p_{122})}{
\mu -\tilde{\alpha}_1[\gamma_1]}+\frac{\tilde{\alpha}_1[F_1(u,v)]}{\mu -\tilde{\alpha}_1[\gamma_1]}\Bigr)+\delta_1(t) \rho(l_{12}\beta_1[1]+p_{122})+ F_1(u,v)(t) \\
=&\rho \frac{\gamma _{1}(t)\tilde{\alpha}_1[\delta_1] (l_{12}\beta_1[1]+p_{122})}{
\mu -\tilde{\alpha}_1[\gamma_1]}+\rho\delta_1(t) (l_{12}\beta_1[1]+p_{122})\\
&+\frac{\gamma _{1}(t)}{\mu -\tilde{\alpha}_1[\gamma_1]}\int_{0}^{1}\mathcal{K}^1_{\tilde{A_1}}(s)g_{1}(s)f_{1}(s,u(s),v(s))\,ds +F_{1}(u,v)(t).
\end{align*}%
Since $\mu \geq 1$, we have $\dfrac{1}{\mu -\tilde{\alpha}_1[\gamma_1]}\leq
\dfrac{1}{1-\tilde{\alpha}_1[\gamma_1]}$ and therefore
\begin{align*}
\mu u(t) \leq &\rho \frac{\gamma _{1}(t)\tilde{\alpha}_1[\delta_1] (l_{12}\beta_1[1]+p_{122})}{
1 -\tilde{\alpha}_1[\gamma_1]}+\rho\delta_1(t) (l_{12}\beta_1[1]+p_{122})\\
&+\frac{\gamma _{1}(t)}{1 -\tilde{\alpha}_1[\gamma_1]}\int_{0}^{1}\mathcal{K}^1_{\tilde{A_1}}(s)g_{1}(s)f_{1}(s,u(s),v(s))\,ds +F_{1}(u,v)(t).
\end{align*}%
Taking the supremum of $t$ on $[0,1]$ gives
\begin{eqnarray*}
\mu \rho &\leq &\rho \frac{\| \gamma _{1} \| _{\infty }\tilde{\alpha}_1[\delta_1] (l_{12}\beta_1[1]+p_{122})}{
1 -\tilde{\alpha}_1[\gamma_1]}+\rho\| \delta _{1} \| _{\infty } (l_{12}\beta_1[1]+p_{122})\\
&&+\rho\frac{\| \gamma _{1} \| _{\infty }}{1 -\tilde{\alpha}_1[\gamma_1]}f_{i}^{0,{\rho }}\int_{0}^{1}\mathcal{K}^1_{\tilde{A_1}}(s)g_{1}(s)\,ds +\rho f_{i}^{0,{\rho }}\frac{1}{m_1}.
\end{eqnarray*}%

Using the hypothesis \eqref{eqmestt} we obtain $\mu \rho <\rho .$ This
contradicts the fact that $\mu \geq 1$ and proves the result.
\end{proof}

\subsection{Index on the set $V_{\rho}$}
We give two Lemma about  the index  on a set
$V_{\rho}$.  In the Lemma~\ref{idx0b1} we assume that the nonlinearities $f_1,f_2$ have the same growth.
The idea in the Lemma~\ref{idx0b3} is similar to the one in Lemma 4 of \cite{gipp-nonlin}:  we control the growth of  one nonlinearity $f_i$, at the cost of having 
to deal with a larger domain. For other results on the existence of
solutions with different growth on the nonlinearities see~\cite{precup1, ya1}. 

\begin{lem}
\label{idx0b1} Assume that

\begin{enumerate}
\item[$(\mathrm{I}_{\protect\rho }^{0})$] there exist $\rho >0$ such that
for every $i=1,2$
\begin{equation}
f_{i,(\rho ,{\rho /c})}\Bigl(\frac{c_{\gamma_i}\| \gamma _{i}\|
_{\infty }}{1-\bar{\alpha}_i[\gamma _{i}]} \int_{a_{i}}^{b_{i}}%
\mathcal{K}^i_{\bar{A_i}}(s)g_{i}(s)\,ds+\frac{1}{M_{i}}\Bigr)>1,  \label{eqMest}
\end{equation}%
{} where
\begin{eqnarray*}
f_{1,(\rho ,{\rho /c})} &=&\inf \Bigl\{\frac{f_{1}(t,u,v)}{\rho }%
:\;(t,u,v)\in [ a_{1},b_{1}]\times [ \rho ,\rho /c]\times
[ 0,\rho /c]\Bigr\}, \\
f_{2,(\rho ,{\rho /c})} &=&\inf \Bigl\{\frac{f_{2}(t,u,v)}{\rho }%
:\;(t,u,v)\in [ a_{2},b_{2}]\times [ 0,\rho /c]\times [
\rho ,\rho /c]\Bigr\} \\
\text{and}\ \frac{1}{M_{i}} &=&\inf_{t\in [
a_{i},b_{i}]}\int_{a_{i}}^{b_{i}}k_{i}(t,s)g_{i}(s)\,ds.
\end{eqnarray*}
\end{enumerate}

Then $i_{K}(T,V_{\rho})=0$.
\end{lem}

\begin{proof}
Let $e(t)\equiv 1$ for $t\in [ 0,1]$. Then $(e,e)\in K$. We prove that
\begin{equation*}
(u,v)\neq T(u,v)+\mu (e,e)\quad \text{for }(u,v)\in \partial V_{\rho }\quad
\text{and }\mu \geq 0.
\end{equation*}%
In fact, if this does not happen, there exist $(u,v)\in \partial V_{\rho }$
and $\mu \geq 0$ such that $(u,v)=T(u,v)+\mu (e,e)$. Without loss of
generality, we can assume that for all $t\in [ a_{1},b_{1}]$ we have
\begin{equation*}
\rho \leq u(t)\leq {\rho /c},\newline
\ \min u(t)=\rho \newline
\ \text{and }\newline
\ 0\leq v(t)\leq {\rho /c}.
\end{equation*}%
For $t\in [ a_{1},b_{1}]$, we have
\begin{align*}
 u(t)&=\gamma_1(t)(H_1(\alpha_1[u])+(d_{11}I_1+e_{11}N_1)(u(\tau_1)))+\delta_1(t)L_1(\beta_1[v])+ F_1(u,v)(t)+\mu e(t)\\
&\geq \gamma_1(t)(H_1(\alpha_1[u])+(d_{11}I_1+e_{11}N_1)(u(\tau_1)))+ F_1(u,v)(t)+\mu e(t).
\end{align*}
 Since
$$
\bar{\alpha}_1[u]\leq H_1(\alpha_1[u])+(d_{11}I_1+e_{11}N_1)(u(\tau_1)),
$$
we have
\begin{equation}
u(t)\geq \gamma _{1}(t)\bar{\alpha}_1[u]+F_{1}(u,v)(t)+\mu e(t).  \label{dis3}
\end{equation}%
Applying $\bar{\alpha}_1$ to both sides of \eqref{dis3} gives
\begin{align*}
\bar{\alpha}_1[u]& \geq \bar{\alpha}_1[\gamma _{1}]\bar{\alpha}_1[u]+\bar{\alpha}_1[F_{1}(u,v)]+\mu \bar{\alpha}_1[e].
\end{align*}%
This can be written in the form
\begin{equation*}
(1-\bar{\alpha}_1[\gamma _{1}])\bar{\alpha}_1[u]\geq \bar{\alpha}_1[F_{1}(u,v)]+\mu
\bar{\alpha}_1[e],
\end{equation*}%
that is
\begin{equation*}
\bar{\alpha}_1[u]\geq \frac{\bar{\alpha}_1[F_{1}(u,v)]}{1-\bar{\alpha}_1[\gamma
_{1}]}+\frac{\mu \bar{\alpha}_1[e]}{1-\bar{\alpha}_1[\gamma _{1}]}.
\end{equation*}%
Thus, \eqref{dis3} becomes
\begin{align*}
u(t)\geq & \frac{\gamma _{1}(t)\bar{\alpha}_1[F_{1}(u,v)]}{1-\bar{\alpha}_1[\gamma _{1}]}+\frac{\mu \gamma _{1}(t)\bar{\alpha}_1[e]}{1-\bar{\alpha}_1[\gamma _{1}]}+F_{1}(u,v)(t)+\mu e(t) \\
=& \frac{\gamma _{1}(t)}{1-\bar{\alpha}_1[\gamma _{1}]} \int_{0}^{1}%
\mathcal{K}^1_{\bar{A_i}}(s)g_{1}(s)f_{1}(s,u(s),v(s))\,ds+\frac{\mu \gamma
_{1}(t)\bar{\alpha}_1[e]}{1-\bar{\alpha}_1[\gamma _{1}]} \\
& +\int_{0}^{1}k_{1}(t,s)g_{1}(s)f_{1}(s,u(s),v(s))\,ds+\mu .
\end{align*}%
Then we have, for $t\in [ a_{1},b_{1}]$,
\begin{align*}
u(t)\geq & \frac{c_{\gamma_1}\| \gamma _{1}\| _{\infty }}{1-\bar{\alpha}_1[\gamma _{1}]} \int_{a_{1}}^{b_{1}}\mathcal{K}^1
_{\bar{A_1}}(s)g_{1}(s)f_{1}(s,u(s),v(s))\,ds+\frac{\mu c_{\gamma_i}\| \gamma_{1}\| _{\infty }\bar{\alpha}_1[e]}{1-\bar{\alpha}_1[\gamma _{1}]} \\
& +\int_{a_{1}}^{b_{1}}k_{1}(t,s)g_{1}(s)f_{1}(s,u(s),v(s))\,ds+\mu  \\
\geq & \frac{c_{\gamma_1}\| \gamma _{1}\| _{\infty }}{1-\bar{\alpha}_1[\gamma_{1}]} \int_{a_{1}}^{b_{1}}\mathcal{K}^1
_{\bar{A_1}}(s)g_{1}(s)f_{1}(s,u(s),v(s))\,ds \\
& +\int_{a_{1}}^{b_{1}}k_{1}(t,s)g_{1}(s)f_{1}(s,u(s),v(s))\,ds+\mu .
\end{align*}%
Taking the minimum over $[a_{1},b_{1}]$ gives
\begin{eqnarray*}
\rho=\min_{t\in [ a_{1},b_{1}]}u(t) &\geq &{\rho }f_{1,(\rho ,{\rho /c})}
\frac{c_{\gamma_1}\|\gamma _{1}\| _{\infty }}{1-\bar{\alpha}_1[\gamma _{1}]} \int_{a_{1}}^{b_{1}}\mathcal{K}^1_{\bar{A_1}}(s)g_{1}(s)\,ds+{\rho }
f_{1,(\rho ,{\rho /c})}\frac{1}{M_{1}}+\mu  \\
&=&{\rho }f_{1,(\rho ,{\rho /c})}\Bigl(\frac{c_{\gamma_1}\| \gamma _{1}\|
_{\infty }}{1-\bar{\alpha}_1[\gamma _{1}]} \int_{a_{1}}^{b_{1}}%
\mathcal{K}^1_{\bar{A_1}}(s)g_{1}(s)\,ds+\frac{1}{M_{1}}\Bigr)+\mu .
\end{eqnarray*}%
Using the hypothesis \eqref{eqMest} we obtain $\rho>\rho +\mu $, a contradiction.
\end{proof}

\begin{lem}\label{idx0b3}
 Assume that
\begin{enumerate}
\item[$(\mathrm{I}_{\protect\rho }^{0})^{\star }$] there exist $\rho >0$
such that for some $i=1,2$
\begin{equation}
f_{i,(0,{\rho /c})}^{\ast }\Bigl(\frac{c_{\gamma_i}\| \gamma _{i}\|
_{\infty }}{1-\bar{\alpha}_i[\gamma _{i}]} \int_{a_{i}}^{b_{i}}%
\mathcal{K}^i_{\bar{A_i}}(s)g_{i}(s)\,ds+\frac{1}{M_{i}}\Bigr)>1,  \label{eqMest1}
\end{equation}%
{}
\end{enumerate}

where
\begin{equation*}
f^*_{i,(0,{\rho / c})}=\inf \Bigl\{ \frac{f_i(t,u,v)}{ \rho}:\; (t,u,v)\in
[a_i,b_i]\times[0,\rho/c]\times[0, \rho/c]\Bigr\}.
\end{equation*}
Then $i_{K}(T,V_{\rho})=0$.
\end{lem}

\begin{proof}
Suppose that the condition \eqref{eqMest1} holds for $i=1$. Let $e(t)\equiv
1 $ for $t\in [ 0,1]$. Then $(e,e)\in K$. We prove that
\begin{equation*}
(u,v)\neq T(u,v)+\mu (e,e)\quad \text{for }(u,v)\in \partial V_{\rho }\quad
\text{and }\mu \geq 0.
\end{equation*}%
In fact, if this does not happen, there exist $(u,v)\in \partial V_{\rho }$
and $\mu \geq 0$ such that $(u,v)=T(u,v)+\mu (e,e)$. So, for all $t\in
[ a_{1},b_{1}]$, $\min u(t)\leq \rho $ and for $t\in [
a_{2},b_{2}]$, $\min v(t)\leq \rho $. 
We obtain, for $t\in [ a_1,b_1]$, with  the same proof of Lemma \ref{idx0b1},
\begin{align*}
u(t)\geq & \frac{\gamma _{1}(t)}{1-\bar{\alpha}_{1}[\gamma _{1}]}
\int_{a_1}^{b_1}\mathcal{K}^1_{\bar{A_1}}(s)g_{1}(s)f_{1}(s,u(s),v(s))\,ds\\
& +\int_{a_1}^{b_1}k_{1}(t,s)g_{1}(s)f_{1}(s,u(s),v(s))\,ds+\mu .
\end{align*}%
Then we have
\begin{eqnarray*}
\min_{t\in [ a_{1},b_{1}]}u(t) &\geq &{\rho }f_{1,(0,{\rho /c})}^{\ast
}\frac{c_{\gamma_1}\| \gamma _{1}\| _{\infty }}{1-\bar{\alpha}_{1}[\gamma_{1}]} \int_{a_{1}}^{b_{1}}\mathcal{K}^1_{\bar{A_1}}(s)g_{1}(s)\,ds+{\rho }f_{1,(0,{\rho /c})}^{\ast }\frac{1}{M_{1}}+\mu. \\
 \end{eqnarray*}%
Using the hypothesis \eqref{eqMest1} we obtain $\min_{t\in [
a_{1},b_{1}]}u(t)>\rho +\mu \geq \rho $, a contradiction.
\end{proof}
\section{Existence and multiplicity of the solutions}
By combining the above results on the index of the sets $V_\rho$ and $K_\rho$  we obtain the following Theorem, in which we
deal with the existence of at least one, two or three solutions. It is
possible to state results for four or more positive solutions by expanding the lists in conditions $(S_{5}),(S_{6})$, see for
example the paper~\cite{kljdeds} for this type of results. \\
 We omit the proof of the Theorem \ref{thmmsol1} which follows from the properties of fixed point index.

\begin{thm}
\label{thmmsol1} The system \eqref{syst} has at least one positive solution
in $K$ if either of the following conditions hold.

\begin{enumerate}

\item[$(S_{1})$] There exist $\rho _{1},\rho _{2}\in (0,\infty )$ with $\rho
_{1}/c<\rho _{2}$ such that $(\mathrm{I}_{\rho _{1}}^{0})\;\;[\text{or}\;(%
\mathrm{I}_{\rho _{1}}^{0})^{\star }],\;\;(\mathrm{I}_{\rho _{2}}^{1})$ hold.

\item[$(S_{2})$] There exist $\rho _{1},\rho _{2}\in (0,\infty )$ with $\rho
_{1}<\rho _{2}$ such that $(\mathrm{I}_{\rho _{1}}^{1}),\;\;(\mathrm{I}%
_{\rho _{2}}^{0})$ hold.
\end{enumerate}

The system \eqref{syst} has at least two positive solutions in $K$ if one of
the following conditions hold.

\begin{enumerate}

\item[$(S_{3})$] There exist $\rho _{1},\rho _{2},\rho _{3}\in (0,\infty )$
with $\rho _{1}/c<\rho _{2}<\rho _{3}$ such that $(\mathrm{I}_{\rho
_{1}}^{0})\;\;[\text{or}\;(\mathrm{I}_{\rho _{1}}^{0})^{\star }],\;\;(%
\mathrm{I}_{\rho _{2}}^{1})$ $\text{and}\;\;(\mathrm{I}_{\rho _{3}}^{0})$
hold.

\item[$(S_{4})$] There exist $\rho _{1},\rho _{2},\rho _{3}\in (0,\infty )$
with $\rho _{1}<\rho _{2}$ and $\rho _{2}/c<\rho _{3}$ such that $(\mathrm{I}%
_{\rho _{1}}^{1}),\;\;(\mathrm{I}_{\rho _{2}}^{0})$ $\text{and}\;\;(\mathrm{I%
}_{\rho _{3}}^{1})$ hold.
\end{enumerate}

The system \eqref{syst} has at least three positive solutions in $K$ if one
of the following conditions hold.

\begin{enumerate}
\item[$(S_{5})$] There exist $\rho _{1},\rho _{2},\rho _{3},\rho _{4}\in
(0,\infty )$ with $\rho _{1}/c<\rho _{2}<\rho _{3}$ and $\rho _{3}/c<\rho
_{4}$ such that $(\mathrm{I}_{\rho _{1}}^{0})\;\;[\text{or}\;(\mathrm{I}%
_{\rho _{1}}^{0})^{\star }],$ $(\mathrm{I}_{\rho _{2}}^{1}),\;\;(\mathrm{I}%
_{\rho _{3}}^{0})\;\;\text{and}\;\;(\mathrm{I}_{\rho _{4}}^{1})$ hold.

\item[$(S_{6})$] There exist $\rho _{1},\rho _{2},\rho _{3},\rho _{4}\in
(0,\infty )$ with $\rho _{1}<\rho _{2}$ and $\rho _{2}/c<\rho _{3}<\rho _{4}$
such that $(\mathrm{I}_{\rho _{1}}^{1}),\;\;(\mathrm{I}_{\rho
_{2}}^{0}),\;\;(\mathrm{I}_{\rho _{3}}^{1})$ $\text{and}\;\;(\mathrm{I}%
_{\rho _{4}}^{0})$ hold.
\end{enumerate}
\end{thm}
We illustrate the conditions that occur in the above Theorem in the following example, where multi-point type BCs are considered.

\begin{ex}
Consider the system
\begin{gather}
\begin{aligned}  \label{eq3.1}
u''+\frac{1}{8}(u^3+t^3v^3)+ 2=0,\,\,\,\,v''=\frac{1}{8}(\sqrt{tu}+13v^2)&,\ t \in (0,1), \\
\Delta u|_{t=1/5}=I_1(u(1/5)), \;\,\Delta u'|_{t=1/5}=N_1(u(1/5))&,\\
\Delta v|_{t=2/5}=I_2(v(2/5)), \;\,\Delta v'|_{t=2/5}=N_2(v(2/5))&,\\
u(0)=H_{1}(u(1/4)),\ u(1)=L_{1}(v(3/4)),  v(0)=H_{2}(v(1/3)),\ v^{\prime}(1)&=L_{2}(u(2/3)).\\
\end{aligned}
\end{gather}

This differential system can be rewritten in the integral form
\begin{align*}
u(t)=&(1-t)H_{1}(u(1/4))+tL_{1}(v(3/4))+G_1(u)(t)+ \int_{0}^{1}k_1(t,s)g_1(s)f_1(s,u(s),v(s))\,ds, \\
v(t)=&H_{2}(v(1/3))+tL_{2}(u(2/3))+G_2(v)(t)+\int_{0}^{1}k_2(t,s)g_2(s)f_2(s,u(s),v(s))\,ds,%
\end{align*}
where the Green's functions
\begin{equation*}
k_1(t,s)=%
\begin{cases}
s(1-t),\, & s \leq t, \\
t(1-s),\, & s>t,%
\end{cases}
\quad \text{and}\quad k_2(t,s)=
\begin{cases}
s,\,  & s\leq t, \\
t,\,  & s>t,%
\end{cases}%
\end{equation*}
are non-negative continuous functions on $[0,1]\times[0,1]$. Here $\gamma_1(t)=1-t$,  $\gamma_2(t)=1$,  $\delta_1(t)=t$,   $\delta_2(t)=t$, 
$c_{\gamma_1}=1-b_1$, $c_{\gamma_2}=1$, $c_{\delta_1}=a_1$ and $c_{\delta_2}=a_2$.
The intervals $[a_1,b_1]$  may be chosen arbitrarily in $%
(1/5,1) $ and $[a_2,b_2]$ can be chosen arbitrarily in $%
(2/5,1]$. It is easy to check that
\begin{equation*}
k_1(t,s) \leq s(1-s):=\Phi_1(s), \quad \min_{t \in [a_1,b_1]}k_1(t,s) \geq
c_{\Phi_1} s(1-s),
\end{equation*}
where $c_{\Phi_1}=\min\{1-b_1,a_1\}$. Furthermore  we have that
\begin{equation*}
k_2(t,s) \leq s:=\Phi_2(s), \quad    \min_{t \in [a_2,b_2]} k_2(t,s) \geq c_{\Phi_2} \Phi_2(s),
\end{equation*}
where $c_{\Phi_2}=a_2$.
The choice $[a_1,b_1]=[1/4,3/4]$ and   $ [a_2,b_2]=[1/2,1]$ gives
\begin{equation*}
c= \frac{1}{4},\,m_1=8,\, M_1=16,\,m_2=2,\, M_2=4.
\end{equation*}
In our example, the nonlinearities  used to illustrate the constants that occur in our theory are taken in a similar way as in ~\cite{gipp-nonlin,  gifmpp-cnsns, gipp-mmas, gippmz}. We consider
$$
H_{1}(w)=\left\{
\begin{array}{l}
\frac{5}{6}w,\;\;0 \leq w \leq 1,\\
\frac{1}{3}w+\frac{1}{2},\;\; w \geq 1,
\end{array}
\right.
\quad
L_{1}(w)=\frac{1}{30}(1+\sin (w ) ),
$$

$$
H_{2}(w)=\left\{
\begin{array}{l}
\frac{1}{19}w,\;\;0 \leq w \leq 2,\\
\frac{1}{25}w+\frac{12}{475},\;\; w \geq 2,
\end{array}
\right.
\quad
L_{2}(w)=\frac{1}{38}(1+\cos (w ) ).
$$

The functions $H_{i}$ and $L_{i}$ 
satisfy the conditions
\begin{equation*}
 h_{i1}w \leq H_{i}(w)\leq  h_{i2}w,\,\,\,\,  L_{i}(w)\leq  l_{i2}w,
\end{equation*}
with
$$
h_{11}=\frac{1}{3}, h_{12}=\frac{5}{6}, h_{21}=\frac{1}{25}, h_{22}=\frac{1}{19}, 
l_{12}=\frac{1}{15}, l_{22}=\frac{1}{75}.
$$

The functions 
\begin{align*}
I_1(w)&=\left\{
\begin{array}{l}
\frac{1}{100} w,\;\;0 \leq w \leq 1,\\
\frac{13}{1400}w+\frac{1}{1400},\;\; w \geq 1,
\end{array}
\right.
\quad
N_1(w)=\left\{
\begin{array}{l}
-\frac{3}{100} w,\;\;0 \leq w \leq 1,\\
-\frac{39}{1400}w-\frac{3}{1400},\;\; w \geq 1,
\end{array}
\right.
\\
I_2(w)&=\left\{
\begin{array}{l}
\frac{1}{300} w,\;\;0 \leq w \leq 1,\\
\frac{1}{400}w+\frac{1}{1200},\;\; w \geq 1,
\end{array}
\right.
\quad
N_2(w)=\left\{
\begin{array}{l}
-\frac{1}{30} w,\;\;0 \leq w \leq 1,\\
-\frac{1}{40}w-\frac{1}{120},\;\; w \geq 1,
\end{array}
\right.
\end{align*}
 satisfy the conditions for $w\in [0,\infty)$
 $$
 p_{i11} w\leq (d_{i1}I_i+e_{i1}N_i)(w) \leq p_{i12} w,
\quad
 0\leq (d_{i2}I_i+e_{i2}N_i)(w) \leq p_{i22} w,
$$
with
$$
d_{11}=1, d_{21}=1, e_{11}=-\frac{1}{5}, e_{21}=-\frac{2}{5}, d_{12}=-1, d_{22}=0, e_{12}=-\frac{4}{5}, e_{22}=-1,  
$$
$$
 p_{111}=\frac{1}{70}, p_{112}=\frac{1}{50}, p_{122}=\frac{1}{40}, p_{211}=\frac{1}{80},  p_{212}=\frac{1}{60}, p_{222}=\frac{1}{30}.
$$
We have that 
$$
\tilde{\alpha}_{1}[\gamma_{1}]=\frac{641}{1000},\tilde{\alpha}_{2}[\gamma_{2}]=\frac{79}{1140}, \tilde{\alpha}_{1}[\delta_{1}]=\frac{634}{3000},\tilde{\alpha}_{2}[\delta_{2}]=\frac{23}{950},  
$$
$$
\bar{\alpha}_{1}[\gamma_{1}]=\frac{183}{700},\bar{\alpha}_{2}[\gamma_{2}]=\frac{21}{400}, \beta_{1}[1]=\beta_{2}[1]=1.
$$

\begin{equation*}
\int_{0}^{1}\mathcal{K}^1_{\tilde{A_1}}(s)\,ds=\frac{3189}{40000}, \int_{0}^{1}\mathcal{K}^2_{\tilde{A_2}}(s)\,ds= \frac{853}{42750},
 \int_{1/4}^{3/4}\mathcal{K}^1_{\bar{A_1}}(s)\,ds=\frac{181}{8400}, \int_{1/2}^{1}\mathcal{K}^2_{\bar{A_2}}(s)\,ds=\frac{11}{1200}.
\end{equation*}

The existence of multiple solutions of the system~\eqref{eq3.1} follows from Theorem~\ref{thmmsol1}.
Then, for $\rho_1=1/8$, $\rho_2=1$ and $\rho_3=11$, we have (the constants that follow have been rounded to 2 decimal places unless exact)
\begin{align*}
\inf \Bigl\{ f_1(t,u,v):\; (t,u,v)\in [1/4,3/4]\times[0,1/2]\times[0, 1/2]%
\Bigr\}&=f_1(1/4,0,0)>14.33 \rho_1, \\
\sup \Bigl\{ f_1(t,u,v):\; (t,u,v)\in[0,1]\times[0, 1]\times[0, 1]\Bigr\}%
&=f_1(1,1,1)<2.46 \rho_2, \\
\sup \Bigl\{ f_2(t,u,v):\; (t,u,v)\in[0,1]\times[0,1]\times[0,1]\Bigr\}%
&=f_2(1,1,1)<1.82 \rho_2, \\
\inf \Bigl\{ f_1(t,u,v):\; (t,u,v)\in [1/4,3/4]\times[11,44]\times[0, 44]%
\Bigr\}&=f_1(1/4,11,0)>14.33  \rho_3, \\
\inf \Bigl\{ f_2(t,u,v):\; (t,u,v)\in [1/2,1]\times[0,44]\times[11, 44]%
\Bigr\}&=f_2(1/2,0,11)>3.86 \rho_3,
\end{align*}
that is the conditions $(\mathrm{I}^{0}_{\rho_{1}})^{\star}$, $(\mathrm{I}%
^{1}_{\rho_{2}})$ and $(\mathrm{I}^{0}_{\rho_{3}})$ are satisfied; therefore
the system~\eqref{eq3.1} has at least two positive solutions in $K$.
\end{ex}

\end{document}